\documentclass[11pt]{amsart}
\usepackage{amsmath}
\usepackage{amssymb}
\usepackage{amscd}
\usepackage[latin2]{inputenc}
\usepackage[T1]{fontenc}

\newcommand{\CC}{\mathbb{C}}

\newcommand{\D}{\mathcal{D}}
\newcommand{\E}{\mathcal{E}}
\newcommand{\en}{ext\,\nu}
\newcommand{\ten}{ext\,\tilde{\nu}}

\DeclareMathOperator{\ord}{ord}

\DeclareMathOperator{\Pic}{Pic}
\DeclareMathOperator{\mx}{max}
\theoremstyle{definition}
\newtheorem{example}{Example}[section]
\newtheorem{definition}[example]{Definition}

\theoremstyle{plain}
\newtheorem{lemma}[example]{Lemma}
\newtheorem{proposition}[example]{Proposition}

\newtheorem{mtheorem}{Theorem}
\newtheorem{conjecture}[example]{Conjecture}
\newtheorem{corollary}[example]{Corollary}

\theoremstyle{remark}
\newtheorem*{acknowledgements}{Acknowledgements}
\newtheorem{remark}[example]{Remark}
\newtheorem{convention}[example]{Convention}
\numberwithin{equation}{section}

\begin{document}
\title[Number of singular points]{Number of singular points of a genus $g$ curve with one point at
infinity}
\author{Maciej Borodzik}
\address{Institute of Mathematics, University of Warsaw, ul. Banacha 2,
02-097 Warsaw, Poland}
\email{mcboro@mimuw.edu.pl}
\date{\today}
\subjclass{primary: 14H50, secondary: 14H20, 32S50, 14E15}
\keywords{singular points, codimension, Zaidenberg--Lin conjecture}
\thanks{Supported by Polish KBN Grant No 2 P03A 010 22}
\begin{abstract} 
Using techniques developped in \cite{BZ0} and \cite{BZ2} we bound the maximal number $N$ of singular points
of a plane algebraic curve $C$ that has precisely one place at infinity with one branch in terms of
its first Betti number $b_1(C)$. Asymptotically we prove that $N<\sim\frac{17}{11}b_1(C)$ for large $b_1$.
In particular, in the case of curves with one place at infinity, 
we confirm the Zaidenberg and Lin conjecture stating that $N\le 2b_1+1$.
\end{abstract}
\maketitle

\section{Introduction}
\subsection{Presentation of results}
In \cite{BZ3} the authors have proved that if $C\subset\mathbb{C}^2$ is an algebraic curve homeomorphic
to $\mathbb{C}^*$ then $C$ may have at most three singular points at finite distance. This confirms 
the following conjecture, due to Zajdenberg and Lin
\begin{conjecture}[Zajdenberg--Lin conjecture, \cite{ZL}]
If $C\subset\mathbb{C}^2$ is a reduced, irreducible, algebraic
plane curve with first Betti number $b_1$, then $C$ has at most $2b_1+1$ singular points at finite distance.
\end{conjecture}

In this article
we use essentially the same techniques as in \cite{BZ3} 
to study the number of singular points of curves of arbitrary genus (by genus we always mean geometric genus)
$g$ with one place at infinity. We confirm
Zajdenberg--Lin conjecture in this case, providing an assymptotically (as $b_1$ goes to infinity), better
estimates.

The condition for a curve to have one place at infinity, makes the computational part  much simpler than in 
the case of general curves. 
We are convinced that the methods presented in this paper can be applied to prove Zajdenberg--Lin conjecture in its full
generality. However the computations without additional assumptions seem to be rather lengty.

The methods developped in \cite{BZ0} and \cite{BZ2}
allow us to prove two following theorems, which are the core of this article.
Before we state them, let us agree on a convention that shall be used throughout the paper.
\begin{convention}
In the whole paper, unless specified otherwise, 
$C$ will denote an algebraic curve in $\CC P^2$ intersecting the line at infinity $L_\infty$
precisely at one point $z_\infty$, such that $C$ has one branch at $z_\infty$. The affine part $C\cap\CC^2$
will be denoted by $C^0$. When talking about genera (arithmetic and algebraic) of the curve, we will have always
in mind the genera of $C$. The first Betti numbers of $C$ and $C^0$ are equal (by Mayer--Viettoris sequence).
\end{convention}

\begin{mtheorem}
Let $C^0\subset\mathbb{C}^2$ be a plane algebraic curve with one place at infinity with $N$ singular points
at finite distance. Assume that all these points are cuspidal. Let $g>0$ denote the geometric genus of
the curve. Then it is impossible that all the following inequalities
hold at the same time:
\begin{subequations}
\begin{align}
N&\ge I_a(g):=2g+3\label{eq:MT1A}\\
N&\ge I_b(g):=\frac{24}{11}g+\frac{20}{11}\label{eq:MT1B}\\
N&>I_c(g):=2g+\frac23+\sqrt{\frac{20}{3}g+\frac{28}{9}}\label{eq:MT1C}\\
N&>I_d(g):=2g-\frac14+\sqrt{7g+\frac{177}{16}}\label{eq:MT1D}\\
N&>I_e(g):=\frac{36}{17}g+\frac{18}{17}\label{eq:MT1E}\\
N&>I_f(g):=\frac{29}{14}g+\frac{31}{28}+\label{eq:MT1F}\\
&+\sqrt{\frac{1}{196}g^2+\frac{1067}{196}g+\frac{793}{784}}.\notag
\end{align}
\end{subequations}
\end{mtheorem}
In other words we have
\begin{corollary}
Let $I(g)=\max(I_a(g),I_b(g),I_c(g),I_d(g),I_e(g),I_f(g))$. Then for any cuspidal curve $C^0$ with geometric genus $g$
and one place at infinity, the number of singular points at finite distance does not exceed $I(g)$.
\end{corollary}
Before we state the second theorem we need one technical definition.
\begin{definition}\label{NoBr}
Let $C\subset X$ be a projective curve.
The \emph{total number of branches} of $C$ is the sum
\begin{equation}\label{eq:defR}
R=\sum_{i\in I} (r_i-1),
\end{equation}
where $i$ goes through all singular points of $C$ and $r_i$ is the total number of branches of $C$ at the $i$-th
singular point. 
\end{definition}
\begin{mtheorem}
Let $C^0$ be a plane algebraic curve with one place at infinity and $N$ singular points at finite distance.
Let $g$ be the geometric genus. Moreover let $R>0$ be the total number of branches of $C$.
Then it is impossible that all the following inequalities hold at the same time: 
\begin{subequations}
\begin{align}
N&\ge J_a(g,R):=2g+3+R\label{eq:MT2A}\\
N&\ge J_b(g,R):=\frac{24}{11}g+\frac{17}{11}R+\frac{18}{11}\label{eq:MT2B}\\
N&>J_c(g,R):=2g+R+\frac23+\sqrt{\frac{20}{3}g+4R+\frac{22}{9}}\label{eq:MT2C}
\end{align}
\begin{align}
N&>J_d(g,R):=2g+R-\frac14+\sqrt{7g+5R+\frac{177}{16}}\label{eq:MT2D}\\
N&\ge J_e(g,R):=\frac{36}{17}g+\frac{23}{17}R+\frac{18}{17}\label{eq:MT2E}\\
N&>J_f(g,R):=\frac{29}{14}g+\frac{17}{14}R+\frac{31}{28}+\label{eq:MT2F}\\
&+\sqrt{\frac{(g+3R)^2}{196}+\frac{1067}{196}g+\frac{513}{196}R+\frac{793}{784}}\notag
\end{align}
\end{subequations}
\end{mtheorem}
\begin{corollary}
Let $J(g,R)=\max(J_a(g,R),J_b(g,R),J_c(g,R)$, $J_d(g,R)$, $J_e(g,R)$, $J_f(g,R))$. 
Then for any cuspidal curve $C^0$ with geometric genus $g$
and one place at infinity, the number of singular points at finite distance does not exceed $J(g,R)$.
\end{corollary}
\begin{remark}
The case $g=0$ from Theorem 1 results from the famous Zaidenberg--Lin theorem. Any curve homeomorphic to a disk can have
at most one singular point at finite distance. For $g=0$ and $R=1$ there is \cite{BZ0}
a conjectural classification of all such curves. All found cases have at most three singular points.
\end{remark}

The structure of this article is the following. We end this section with a discussion of the behaviour of quantities $I(g)$
and $J(g,R)$.
In fact, we have formulated the theorems above to make the proof as transparent as possible. It remains to show
how all these inequalities are related each to other. This gap we shall fill in Section~\ref{S12}.

In Section~\ref{BI} we recall definitions of some non classical invariants of planar singular points as a codimension. We
prove several bounds relating multiplicities, $\delta$-invariants and codimensions of singular points to the genus,
degree and the total number of branches of the curve $C$. We turn the attention of the reader to Proposition~\ref{newProp} which
is an important generalisation of known inequalities and is of interest on its own.

Section~\ref{BI} closes with a set of inequalities from which we will deduce
Theorems 1 and 2.

The proof of Theorem 1 is contained in Section~\ref{Sec3}. It is considerably simpler than the proof of Theorem 2, which
is comprised in Section~\ref{Sec4}.

Altough these proofs are highly technical, it is worth to mention that the method, which informally can be resumed as 
``write down all the inequalities and see what happens'', proves very efficient in the study of affine algebraic curves.

\subsection{Inequalities in Theorems 1 and 2}\label{S12}
First let us do the easy case of quantities $I_a(g),\dots I_f(g)$. We see that the leading term
in $g$ is $\frac{24}{11}g$ and appears in $I_b(g)$. Therefore, asymptotically, $I(g)=I_b(g)$.

We observe that always $I_e(g)<I_b(g)$. For $g>0$, the inequality \eqref{eq:MT1A}
is not stronger than \eqref{eq:MT1B}. In fact, we have $I_a(g)\le I_b(g)$ for $g\ge 4$,
but for $g=1,2,3$ the both inequalities are equivalent, as $N$ is assumed to be an integer.

Let us write $L_b=\frac{24}{11}g+\frac{20}{11}$, $L_c=2g+\frac23$, $L_d=2g-\frac14$, $L_f=\frac{29}{14}g+\frac{31}{28}$.
Moreover, let $R_c=\frac{20}{3}g+\frac{28}{9}$, $R_d=7g+\frac{177}{16}$ and 
$R_f=\frac{g^2}{196}+\frac{1067}{196}g+\frac{793}{784}$.
\begin{lemma}
We have $I(g)=I_b(g)$, if $g\ge 747$.
\end{lemma}
\begin{proof}
We shall prove only that for $g\ge 747$, $L_b=I_b(g)\ge L_f+\sqrt{R_f}=I_f(g)$. We have
\[(L_b-L_f)^2-R_f=\frac{6}{847}(g^2-\frac{2239}{3}g-\frac{875}{12}).\]
The polynomial on the right hand side has two roots
$x=\frac{2239\pm \sqrt{5015746}}{6}$. The larger one is approximately $746.4$. The other inequalities are
proven similarly.
\end{proof}
For $g<747$ we can bound $I(g)$ for
example by
\[I(g)\le 3g+\frac32,\,\,\, 2.4g+6,\,\,\,2.2g+20.\]
These bounds have been found experimentally using computer. They are not optimal, but linear in $g$ and
therefore much easier to handle. In particular, as the first Betti number $b_1(C^0)$
is equal to $2g$, we prove
\begin{corollary}
A curve $C^0$ of geometric genus $g>1$ without self--in\-ter\-sec\-tions at finite distance and 
having one point at infinity cannot have
more than $2b_1(C^0)+1$ singular points. 
\end{corollary}
We confirm the Zaidenberg--Lin conjecture in this case.
As for the case $g=1$, the inequality \eqref{eq:MT1D} does not hold for $N=6$, so curves with
six singular points are apparently allowed. See Remark~\ref{ong=1} for detailed discussion.

Things are obviously more complicated in the case $R>0$. Again the quantity $J_b(g,R)$ is assymptotically
the largest both in $g$ and $R$ so a curve with large arithmetic genus $g+R$ 
cannot have more than $\frac{24}{11}g+\frac{17}{11}R+\frac{18}{11}$
singular points, whereas the first Betti number is equal to $2g+R$. However for curves of smaller genus, the quantities
as $J_d(g,R)$ or $J_f(g,R)$ may be larger.
\begin{lemma}
For $g\ge 752-3R$ we have $J(g,R)=J_b(g,R)$.
\end{lemma}
\begin{proof}
We shall only prove that $J_b(g,R)\ge J_f(g,R)$. This amounts to the fact that
\begin{multline*}
\left(\frac{24}{11}g+\frac{17}{11}R+\frac{18}{11}-\frac{29}{14}g-\frac{17}{14}R-\frac{31}{28}\right)^2\ge\\
\ge\frac{(g+3R)^2}{196}+\frac{1067}{196}g+\frac{513}{196}R+\frac{793}{784}.
\end{multline*}
After straighforward transformations we get
\[
(g+3R-376)^2\ge 141479.25-1936R.
\]
But if $g+3R\ge 752$, the right hand side at least $141376$. As $R\ge 1$ this is larger than the left hand side.
\end{proof}
Using a simple computer program we can check, that for all integer values $g\ge 0$ and $R\ge 1$ such that $g+3R\le 752$
(there are only finitely many of them), 
with exception of $g=0,R=1$ and $g=0,R=2$,
we have $[J(g,R)]\le 4g+2R+1$, where $[\cdot]$ denote the integer part. This confirms the Zaidenberg--Lin conjecture
in these cases. As for the case $g=0,R\le 2$, see Remark~\ref{ong=0}.

\section{Basic inequalities}~\label{BI}

\subsection{Bounding multiplicities}\label{ss1}
Let $z_1,\dots,z_N$ be singular points of a planar algebraic curve $C$. Let $z_\infty\subset\CC P^2$ be
the point of $C$ at infinity. In the whole paper we assume that $C$ has only one place at infinity and precisely
one branch at that point.

Let $m_i$ be the multiplicity of $z_i$, $\delta_i$ the $\delta$-invariant of $z_i$ and $r_i$ the number
of branches of $C$ at $z_i$.

Let $\nu:\Sigma\to C$ be the normalisation map. The composition $\Sigma\to C\hookrightarrow\mathbb{C}^2$
is given by two meromorphic functions $x$ and $y$. $x$ and $y$ have precisely one pole, let us call it 
$t_\infty$. Let $p$ and $q$ be orders of poles $x$ and $y$ respectively at $t_\infty$. We can always assume that
$p<q$ and $p\!\!\!\not|q$. Otherwise, if $p|q$ we apply a de 
Jonqui\`ere automorphism $y\to y-\text{const}\cdot x^{q/p}$ and reduce
the order of $y$. Note that $q$ is the degree of $C$.

Now for $i=1,\dots,N$, let $t_{i1},\dots,t_{ir_i}$ be the inverse images of $z_i$ under $\nu$. If $x$ has
order $n_{ij}$ at the point $t_{ij}$, i.e. $x(t)=x(t_{ij})+O((t-t_{ij})^{n_{ij}})$, then
the multiplicity of $C$ at $z_i$ is at least $\sum_{j=1}^{r_i}n_{ij}$. Therefore
\[\sum_{j=1}^{r_i}\ord_{t_{ij}}dx\ge m_i-r_i,\]
where $\ord$ denotes the order of zero of the meromorphic form $dx$. But $dx$ has only one pole, at $t_\infty$, and the
order of this pole is $p+1$. Using standard arguments from geometry of curves we infer that
\begin{equation}
\sum_{i=1}^N(m_i-r_i)\le p+2g-1.
\label{eq:multbound}
\end{equation}
This is the simplest bound we are going to use. Nevertheless it proves very important. At the beginning we will use it
to bound the multiplicities of the singular points. But later on, we shall also bound $p$ using this inequality.

\subsection{Bounding Milnor numbers}\label{BoMiNu}
To obtain a second bound we have to recall the following lemma from \cite[Proposition 2.11]{BZ0}
\begin{lemma}\label{Mil1}
Let us be given a germ of a planar cuspidal singularity $(A,0)$ parametrised locally by 
\begin{equation}\label{eq:form1}
x=t^n+\dots,\,\,\,y=t^m+\dots.
\end{equation}
Then the Milnor number of number $\mu(A)$ of this singularity can be written as
\begin{equation}
\label{eq:Milnor1}
\mu(A)=(n-1)(m-1)+n'-1+\mu'(A),
\end{equation}
where $n'=\gcd(n,m)$ $\mu'(A)\ge 0$. Moreover, 
if $x$ and $y$ are generic in the space of all convergent series of the form \eqref{eq:form1}
then $\mu(A)=0$.
\end{lemma}
Let us apply this lemma to the singularity of $C$ at infinity. We obtain the fact that the Milnor number at infinity
is equal to $(q-1)(q-p-1)+p'-1+\mu'_\infty$, where $p'=\gcd(p,q)$. Here we use essentially the assumption that
$C$ has precisely one place at infinity. As the singularity at infinity is unibranched, its Milnor number is twice
the $\delta$-invariant of the singularity. We recall that the $\delta$-invariant measures the number of double points
hidden at a given singular point (see \cite{BZ0}).

Recall that $C$ is a curve of degree $q$ and geometric genus $g$. By genus formula
\[(q-1)(q-2)-\sum_{i=1}^N 2\delta_i-2\delta_\infty=2g.\]
We sum up all the $\delta-$invariants at finite distance and the $\delta-$invariant at infinity. By Lemma~\ref{Mil1}
we obtain
\begin{equation}\label{eq:Milnor2}
\sum_{i=1}^N2\delta_i+\mu'_\infty=(p-1)(q-1)-p'+1-2g\stackrel{def}{=}\D.
\end{equation}
The quantity $\D$ will be called the \emph{number of double points at finite distance}. A generic curve $C$ (i.e. where
$x(t)$ and $y(t)$ are generic meromorphic functions with prescribed order of pole at $t_\infty$) has
only double points as its singularities at finite distance. The number of these double points is $\D/2$.

Below we shall bound $2\delta_i$ and $\mu'_\infty$ in terms of multiplicity of a given singular point and so--called
codimension of the singular point. The equality \eqref{eq:Milnor2} will become vital in our estimates. But first 
let us do some local analysis.

\subsection{Codimension of a singular point}
Let $(A,0)\subset (\CC^2,0)$ be a germ of a planar singular curve. Let $\pi:X\to\CC^2$ be the minimal
resolution of the singularity with $E=E_1+\dots+E_s$ the exceptional divisor with reduced scheme structure and
$A'$ the strict transform of $A$. In $H^2(X,\mathbb{Q})$ $A'$ is homologous to a linear combination of the divisors $E_i$,
i.e. $A'=\sum\alpha_iE_i$.
Define $K_E=\sum\beta_iE_i$ by the condition that
\[E_i\cdot(K_E+E_i)=-2\textrm{ \,\,\,\,for all $i=1,\dots,s$,}\]
where $\cdot$ denotes the intersection product. $K_E$ is the (local) canonical divisor of $X$.
\begin{definition}[see \cite{BZ2}]
The quantity
\[\en=K_E\cdot(K_E+E+A')\]
is called the \emph{codimension} of the singular point. Moreover, let $K_E+A'+E=P_E+N_E$ be the Zariski--Fujita
decomposition of the divisor $K_E+A'+E$. $P_E$ is the nef part and $N_E$, the negative part. The quantity
\[\eta:=-N_E^2\ge 0\]
is called the \emph{excess} of the singular point.
\end{definition}
\begin{remark}
In \cite{BZ0} and \cite{BZ2} the external codimension is defined in a different way, namely by means of
Puiseux expansion of branches of the given singular point. That definition, in general, depends on
the choice of coordinates near a singular point. Yet, if the coordinates are generic, by Lemma~2.1 and
Proposition~4.1 in \cite{BZ2}, both definitions agree. We should perhabs call the quantity $K_E(K_E+E+A')$,
rough $\bar{M}$-number, but we find the term ``codimension'' to be more geometric.
\end{remark}
We have the following
\begin{lemma}\emph{(}see \cite[Proposition~2.9, 2.16 and Definition~3.6]{BZ0}\emph{)}\label{bz0}
With the notation as above, let $\delta$ be the $\delta$-invariant of $A$ and $m$, the multiplicity.
If $A$ has one branch then
\begin{equation}
\label{eq:br1}
2\delta\le m(\en-m+2).
\end{equation}
If $A$ has two branches, we have
\begin{equation}
\label{eq:br2}
2\delta\le m(\en-m+3).
\end{equation}
\end{lemma}
There similarity of the two formulae is striking. In fact, we can generalise them
as follows
\begin{proposition}~\label{newProp}
With the notation as above, if $r$ denotes the number of branches of $A$ then
\begin{equation}
\label{eq:brn}
2\delta\le m(\en-m+r+1).
\end{equation}
\end{proposition}
\begin{proof}
The proof is very similar to the proof of formula \eqref{eq:br2}, which is contained in \cite[Proposition~2.16]{BZ0}. 
We proceed with induction
with respect to $r$. In order to do this, we have to introduce some additional notation.

For an $r$-branched singularity $A$ with multiplicity $m$, $\delta$-invariant $\delta$ and external codimension $\en$, 
let $A_1,\dots,A_r$ be its irreducible components. Given a branch $A_i$,
let $m_i$ be its multiplicity, $\delta_i$ its $\delta$-invariant, $\en_i$ its codimension and $\nu_i=\en_i-m_i+2$ its
$y$-codimension (see \cite[Definition~2.3]{BZ2}).

Moreover, for two distinct branches $A_i$ and $A_j$ we denote by $\nu_{ij}$ the tangency codimension (see \cite[Remark~1]{BZ2})
of $A_i$ and $A_j$ and by $\varepsilon_{ij}$, the local intersection index of $A_i$ and $A_j$.

We have
\begin{subequations}
\begin{align}
m&=m_1+\dots+m_r\label{eq:s1}\\
\en&=\sum_{i=1}^r\en_{i}+\sum_{j=2}^r\max_{1\le i<j}\nu_{ij}+2r-2&\textrm{ (see \cite{BZ2})}\label{eq:s2}\\
2\delta&=\sum_{i=1}^r2\delta_i+2\sum_{i<j}\varepsilon_{ij}.\label{eq:s3}\\
\intertext{In \cite{BZ0} it was proved that}
2\delta_i&\le m_i\nu_i=m_i(\en_i-m_i+2)\label{eq:s4}\\
\varepsilon_{ij}&\le m_i(\nu_j+\nu_{ij}+1).\label{eq:s5}
\end{align}
\end{subequations}
Up to reordering of first $r-1$ branches we may assume that
\begin{equation}\label{eq:s5.5}
\max_{i\le r-1}\nu_{ir}=\nu_{r-1,r}. 
\end{equation}
Let $\tilde{\delta}$, $\tilde{m}$ and
$\ten$ be corresponding invariants of the singularity $\tilde{A}=A_1+\dots+A_{r-1}$. By \eqref{eq:s2} we have
\begin{equation}\label{eq:s6}
\ten=\sum_{i=1}^{r-1}\en_i+\sum_{j=2}^{r-1}\max_{1\le i\le j-1}\nu_{ij}+2r-4.
\end{equation}
From \eqref{eq:s2} we derive
\begin{equation}\label{eq:s6.5}
\en=\ten+\nu_{r-1,r}+\en_r+2,\,\,\,\,\textrm{cf. \cite[formula (2.9)]{BZ2}}
\end{equation}
The induction assumption that we are making is
\begin{equation}\label{eq:s7}
2\tilde{\delta}\le \tilde{m}(\ten-\tilde{m}+(r-1)+1),
\end{equation}
while we want to prove \eqref{eq:brn}, i.e. $2\delta\le m(\en-m+r+1)$. By \eqref{eq:s3} we have
\begin{equation}\label{eq:s8}
2\delta=2\tilde{\delta}+2\sum_{i=1}^{r-1}\varepsilon_{ir}+2\delta_r.
\end{equation}
From \eqref{eq:s5} and the obvious fact that $\varepsilon_{ir}=\varepsilon_{ri}$ we infer that
\[
2\varepsilon_{ir}\le m_i(\nu_r+\nu_{ir}+1)+m_r(\nu_i+\nu_{ir}+1).
\]
Summing this up and applying \eqref{eq:s5.5} we obtain
\begin{equation}
\label{eq:s9}
2\sum_{i=1}^{r-1}\varepsilon_{ir}\le \tilde{m}(\nu_r+\nu_{r-1,r}+1)+m_r(\sum_{i=1}^{r-1}(\nu_i+\nu_{ir})+r-1).
\end{equation}
By induction assumption \eqref{eq:s7} we obtain from \eqref{eq:s8}, \eqref{eq:s9} and \eqref{eq:s4}
\begin{equation}\label{eq:s10}
\begin{split}
2\delta\le&\tilde{m}\left(\ten-\tilde{m}+r+\nu_r+\nu_{r-1,r}+1\right)+\\
&+m_r\left(\sum_{i=1}^r\nu_i+r-1+\sum_{i=1}^{r-1}\nu_{ir}\right).
\end{split}
\end{equation}
The first term in parenthesis is equal to $\en-m+r+1$ by \eqref{eq:s6.5}. It remains to prove that
\[
\sum_{i=1}^r\nu_r+r-1+\sum_{i=1}^{r-1}\nu_{ir}\le\en-m+r+1. 
\]
But $\nu_i=\en_i-m_i+2$, so the above inequality becomes
\begin{equation}\label{eq:s11}
\sum_{i=1}^r\en_i+\sum_{i=1}^{r-1}\nu_{ir}+2r-2\le \en.
\end{equation}
By the recursive formula \eqref{eq:s2} the inequality \eqref{eq:s11} is equivalent to
\begin{equation}\label{eq:s12}
\sum_{i=1}^{r-1}\nu_{ir}\le\sum_{j=2}^r\max_{1\le i\le j-1}\nu_{ij}.
\end{equation}
This will follow from the following
\begin{lemma}\emph{(}\cite[Lemma~2.13]{BZ2}\emph{)}\label{lemglupglup}
Assume we are given three branches $A$, $B$ and $C$ of one given singular point and let $\nu(A,B)$, $\nu(A,C)$ and
$\nu(B,C)$ denote the corresponding tangency codimensions. If $\nu(A,C)<\nu(A,B)$ then
\[\nu(A,C)=\nu(B,C).\]
\end{lemma}
From this lemma, formula \eqref{eq:s12} follows by induction on $r$. Suppose that
we already know that
\[\sum_{i=1}^{r-2}\nu_{i,r-1}\le \sum_{j=2}^{r-1}\nu_{j-1,j}.\]
We ask whether
\[\sum_{i=1}^r\nu_{ir}\le \sum_{j=2}^r\nu_{j-1,j}.\]
Substracting both sides of these inequalities we see that the induction step shall be accomplished once we have shown that
\[\sum_{i=1}^{r-2}(\nu_{ir}-\nu_{i,r-1})+\nu_{r-1,r}\le \nu_{r-1,r}.\]
Suppose that $\nu_{ir}>\nu_{i,r-1}$ for some $i\le r-2$. By Lemma~\ref{lemglupglup} it follows
that $\nu_{i,r-1}=\nu_{r-1,r}$. But then $\nu_{ir}>\nu_{r-1,r}$ which contradicts \eqref{eq:s5.5}.
\end{proof}
\begin{example}\label{ordinary}
The many inequalities that appear in the above proof suggest that the estimate \eqref{eq:brn} is not optimal
and could be improved. This impression is unfortunately misleading, 
as can be seen by looking on the ordinary $n$-tuple point. Such
a point has external codimension of $n-2$, multiplicity $n$ and $\delta$-invariant $2\delta=(n^2-n)$.
For such singularity we have the equality of both sides of \eqref{eq:brn}.
\end{example}
Let us recall also the estimate for $\mu'(A)$ defined in Lemma~\ref{Mil1}.
\begin{lemma}\emph{(}see \cite[Proposition 2.11]{BZ0}\emph{)}\label{Mil2}
Let $x=t^n+\dots$, $y=t^m+\dots$ give a local parametrisation of an unibranched singularity. Let $\mu'(A)$ be
as in Lemma~\ref{Mil1} and $n'=\gcd(n,m)$. Then 
\[\mu'(A)\le n'\nu',\]
where $\nu'$ is the subtle codimension (see ibidem) of the singular point.
\end{lemma}
Let us return for a while to notation of Section~\ref{ss1}. Let the $i$-th singular point has $r_i$ branches
and codimension $\en_i$. Let the subtle codimension at infinity be denoted by $\nu'_\infty$. Introduce
the notation
\begin{equation}\label{eq:mathcalE}
\E:=\sum_{i=1}^N m_i(\en_i-m_i+r_i+1)+p'\nu'_\infty.
\end{equation}
Then, the equality \eqref{eq:Milnor2} together with inequality~\eqref{eq:brn} and Lemma~\ref{Mil2} yield
\begin{equation}
\Delta=\D-\E\le 0.
\end{equation}
In all instances of the proof of Theorem~1 and Theorem~2, we shall strive to prove that if all the inequalities
\eqref{eq:MT1A},\dots, \eqref{eq:MT1F}
(respectively \eqref{eq:MT2A},\dots, \eqref{eq:MT2F}) hold then always $\Delta>0$.

\subsection{Estimating codimensions}
The bounds of type \eqref{eq:brn} are of no use to us if we cannot control the sum of codimensions of different singular
points of a fixed algebraic curve $C$. However the Bogomolov--Miyaoka--Yau inequality (see \cite{KNS}) can be
used to estimate this sum. Let us recall and slightly generalise the results obtained in \cite{BZ2}.

Recall that we are studying a curve $C\subset\CC P^2$ with $N$ singular points $z_1,\dots,z_N$ at finite distance and
one place at infinity $z_\infty$ with one branch. 
Let $L_\infty$ be the line at infinity. Let $\pi:X\to\CC P^2$ be the minimal resolution
of singularities $\tilde{C}=C\cup L_\infty$ such that $\pi^{-1}(\tilde{C})_{red}$ is an NC divisor. Denote by $C'=\pi'(\tilde{C})$ the
strict transform of $\tilde{C}$ and by $E$, the reduced exceptional divisor. 
Finally, let $K=K_X$ be the canonical divisor and
let us put $D=C'+E$.

We assume that the resolution is minimal, so $E$ does not contain any $(-1)-$curve $F$ such that $F(E-F)\le 2$. We shall
assume also that the pair $(X,D)$ is relatively minimal. There are arguments in \cite{BZ2} that if $(X,D)$ is
not relatively minimal then the estimates obtained are even better. Therefore the relative minimality assumption is used
only to make the discussion below more transparent to the reader.

We are interested only in the case when $\deg C\ge 4$ and $C$ has at least three singular points. 
By \cite{Wa} $\bar{\kappa}(\CC P^2\setminus C)=2$,
so $\bar\kappa(X\setminus D)$ is also equal to $2$. Here $\bar\kappa$ stays for the logarithmic Kodaira dimension.

By the relative minimality and since $\bar\kappa(X\setminus D)\ge 0$, there exists the Zariski--Fujita
decomposition of the divisor $K+D$. Let $H$ denotes the nef part and $N$ the negative part of $K+D$. We have
$K+D=H+N$ and $(K+D)^2=H^2+N^2$ with $N^2<0$. The BMY inequality says that
\[H^2\le 3\chi(X\setminus D)=3-3\chi(C^0).\]
But $C$ is a curve of geometric genus $g$ with total number of branches $R$ (see Definition~\ref{NoBr}). Then the Euler characteristic
of $C^0$ is equal to $1-2g-R$. Therefore we obtain
\[(K+D)^2\le 6g+3R+N^2.\]
Now $(K+D)^2=K(K+D)+D(K+D)=K(K+D)+2p_a(D)-2$, where $p_a(D)$ is arithmetic genus of $D$. By invariance of arithmetic genus
we have $p_a(D)=p_a(C\cup L_\infty)=g+R$. Hence we obtain
\begin{equation}\label{eq:bmy1}
K(K+D)\le 4g+R+2-(-N^2).
\end{equation}
Observe now that $W=\Pic X\otimes\mathbb{Q}$ is spanned by the class of (inverse image of) 
a generic line $L$ in $\CC P^2$ and
by exceptional divisors of the map $\pi$. Let $V_0\subset W$ be the one--dimensional linear subspace
of $W$ spanned by $L$ (we denote the divisor and its class in $W$ by the same letter). Let $V_i\subset W$,
$i\in\{1,2,\dots,N,\infty\}$ be the subspace spanned by exceptional divisors $E_{i1},\dots,E_{is_i}$
such that $\pi(E_{ij})=z_i$. It is easy to see that $W$ is the direct sum of spaces $V_0,V_1,\dots,V_N,V_\infty$
and all components of this sum are pairwise orthogonal with respect to the intersection form. Therefore we can compute
$K(K+D)$ by projecting $K$ and $K+D$ onto spaces $V_i$, computing $K_i(K_i+D_i)$ and summing up the results.

More precisely, let $K=K_0+K_1+\dots+K_N+K_\infty$, $D=D_0+D_1+\dots+D_N+D_\infty$ be the decomposition
of $K$ and $D$ into pieces lying in different subspaces $V_i$. Obviously we have $K_0=-3L$ and $D_0=-(\deg C+1)L$.
Moreover, by construction $K_i(K_i+D_i)=\en_i$ for $i=1,\dots,N$.
\begin{lemma}\emph{(}\cite[Lemma~4.26]{BZ2}\emph{)}
With the notation as above we have
\[K_\infty(K_\infty+D_\infty)=q+q-p-2+\nu'_\infty.\]
\end{lemma}
Therefore we obtain from \eqref{eq:bmy1}
\begin{equation}\label{eq:bmy2}
\sum_{i=1}^N\en_i+\nu'_\infty\le p+q-2+4g+R-(-N^2).
\end{equation}
To complete our task in this section we have to estimate $N^2$. By definition $N$ is supported on all rational twigs of $D$
(see \cite{Fuj}). So $N$ is the sum of components $N_i$ lying in $V_i$ for $i\in\{1,\dots,N,\infty\}$,
$-N^2=-N_1^2-\dots-N_N^2-N_\infty^2$. But $-N_i^2=\eta_i$ is the excess of the singular point $z_i$. 
The inequality \eqref{eq:bmy2}
can be rewritten as
\begin{equation}\label{eq:bmy3}
\sum_{i=1}^N\en_i+\nu'_\infty\le p+q-2+4g-\sum_{i=1}^N\eta_i.
\end{equation}
The term $\eta_\infty$ in \eqref{eq:bmy3} has been omitted. In fact, for a multibranched singularity the only thing
we know about $\eta$ is that it is non--negative. For example, for an ordinary $n-$tuple point, $\eta=0$. On the other
hand we have the following
\begin{lemma}[see \cite{ZO,BZ2}]
For a cuspidal singularity $\eta>\frac12$. Moreover if the multiplicity of the singular point is equal to $2$
then $\eta\ge\frac56$.
\end{lemma}
\begin{remark}
This coefficient $\frac56$ affects very strongly the estimates given in Theorems~1 and~2 (e.g. $\frac{24}{11}=4\cdot(\frac56+1)^{-1}$).
There is a hope that a refinement of the BMY inequality (for example such like in \cite{Lan}) can lead to
an improvement of our results.
\end{remark}
\section{Proof of Theorem 1}\label{Sec3}
We will attempt to pick such $\en_i$, $m_i$ and later $p$ and $q$ so that $\E$ is maximal
possible and then $\Delta$ is as small as possible, ensuring that the inequalities
\eqref{eq:multbound} and \eqref{eq:bmy3} are satisfied.

First we observe that while we are studying the quantity $\E$ (see \eqref{eq:mathcalE}) two possibilities may occur: 
either the term $\sum_{i=1}^N m_i(\en_i-m_i+r_i+1)$ is dominating,
or the term $p'\nu'_\infty$ is dominating. The second possibility may occur 
when $p'>\max m_i$, in particular $p'\ge 3$ (otherwise we can increase, say $\en_1$, at the cost of decreasing $\nu'_\infty$ and
$\E$ will not decrease). Therefore we shall discuss two separate cases.
\subsection{When singularities at finite distance are dominating.}
Upom renumerating $z_1,\dots,z_N$ we may assume that $m_1\ge m_2\ge\dots\ge m_N$.
Observe that, by assumptions of Theorem 1, $r_i=1$. Therefore the inequality
\eqref{eq:multbound} takes the following form
\begin{equation}\label{eq:multbound2}
\sum_{i=1}^N(m_i-1)\le p+2g-1.
\end{equation}
It is easy to see that $\E$ is maximal if $\en_1$ is maximal possible and $\en_i$, $\nu'_\infty$
are minimal possible. We can also assume that $m_i$ are minimal possible for $i\ge 2$.
More
concretely, suppose first that $m_2,\dots,m_{s+1}\ge 3$ and $m_{s+2}=\dots=m_N=2$.
Denoting
\begin{equation}\label{eq:defr}
r=N-1-s
\end{equation}
we arrive at
\begin{equation}\label{eq:E0}
\E=m_1(\en_1-m_1+2)+2r+6s,
\end{equation}
where
\begin{subequations}
\begin{align}
m_1+r+2s\le& p+2g\label{eq:B00}\\
\en_1\le& p+q+4g-\frac52-\frac{11}{6}r-\frac72s.\label{eq:B01}
\end{align}
\end{subequations}
In the second inequality \eqref{eq:B01} we used the fact that there are precisely $r$ singular points
with multiplicity $2$ (so excess is at least $\frac56$ and codimension at least $1$),
and $s$ points with multiplicity $3$, so their codimension is at least $3$. There are $s+1$ points
$m_1,\dots,m_{s+1}$ with excess greater than $\frac12$. As $\E$ grows with $\en_1$, we shall assume
the equality in \eqref{eq:B01}.

Now observe that $\en_1$ is larger than $2(p-2g-2s-r)$. Therefore
the quadratic function $m_1\to m_1(\en_1-m_1+2)$ is increasing for $m_1$ satisfying 
\eqref{eq:B00}. Hence $\E$ is maximal if $m_1$ is maximal possible. 

Therefore we write
\begin{equation}\label{eq:wtracone}
\E\le \E_1=2r+6s+(p+2g-r-2s)(q+2g-\frac56r-\frac32s-\frac12).
\end{equation}
Using $N=r+s+1$ we obtain from the above formula
\[
\E_1=2N+4s-2+(p+2g-s-N+1)(q+2g-\frac56N-\frac23s+\frac13).
\]
Differentiating $\E_1$ with respect to $s$ we get
\[\frac{\partial \E_1}{\partial s}=4+\frac43s-(q+2g-\frac56N+\frac13)-\frac23(p+2g-N+1).\]
As $q\ge p+1$, by \eqref{eq:multbound} we see that the first expression in parenthesis
is greater or equal to $s+\frac16N+\frac73$, while the second is bounded from below by $s+2$.
Therefore $\frac{\partial \E_1}{\partial s}\le \frac13-\frac13s-\frac16N<0$. Hence
\[\E_1\le \E_2:=\E_1|_{s=0}=2N-2+(p+2g-N+1)(q+2g-\frac56N+\frac13).\]
\begin{remark}
We could not have just simply differentiated \eqref{eq:wtracone} with respect to $s$ to obtain that $\E_1$ decreases with $s$.
In fact, if we keep $N$ constant then changing $s$ results in changing also $r$ according to \eqref{eq:defr}.
Therefore, before applying $\frac{\partial}{\partial s}$ we have put $\E_1$ in such a form, that other variables do not
depend implicitly on $s$. This type of reasoning will be used in the sequel without additional comments.
\end{remark}
Now let us define $\Delta_2=\D-\E_2$. As $p'\le q-p$ we can write
\[\Delta_2\ge\Delta_3=(p-1)(q-1)-(p+2g-N+1)(q+2g-\frac56N+\frac13)-2g-2N+p-q+3.\]
The derivative of $\Delta_3$ with respect to $q$ is equal to $-2g+N-3$.
By \eqref{eq:MT1A} $\Delta_3$ is increasing  with respect to $q$. Putting $q=p+1$ we obtain
\begin{align*}
\Delta_3\ge\Delta_4=\Delta_3|_{q=p+1}=&p(p-1)-\left(p+2g-N+1\right)\left(p+2g-\frac56N+\frac43\right)-\\
&-2g-2N+2.
\end{align*}
$\Delta_4$ is a linear function with $p$ with leading coefficient
\[\frac{11}{6}N-4g-\frac{10}{3}.\]
But, by \eqref{eq:MT1B} this term is positive. Therefore $\Delta_4$ will be
minimal if $p$ is minimal possible. In view of \eqref{eq:multbound2} we have $p\ge N+1-2g$. Hence
\[\Delta_4\ge\Delta_5=\Delta_4|_{p=N+1-2g}=N^2-4gN+4g^2-\frac43N-4g-\frac83.\]
From \eqref{eq:MT1C}, $N> 2g+\frac23+\sqrt{\frac{20}{3}g+\frac{28}{9}}$,
so $\Delta_5>0$.

Therefore in the case when terms with $m_1(\en_1-m_1+2)$ is larger than $p'\nu'_\infty$ leads
to $\Delta>0$.

Remark that if, say, \eqref{eq:MT1B} does not hold then $\Delta_4$ is linear function that decreases with $p$.
Therefore we can by no means expect that $\Delta_4>0$ for all reasonable values of $p$: we can take $p$
as large as we want so that $\Delta_4<0$ and nothing can be done. Therefore the inequality \eqref{eq:MT1B} has real importance.
\subsection{When $p'$ dominates.}\label{2.3}

Since $p'\ge 3$ we have also (see beginning of Section~\ref{Sec3}).
\begin{equation}\label{eq:trivial1}
p\ge 6.
\end{equation}

Let us assume that there are $s$ singular points at finite distance with multiplicity $\ge 3$ and $r$ singular points
with multiplicity $2$. Together there are $N=r+s$ singular points at finite distance. Observe that in this case we
do not distinguish a special singular point $m_1$ and formula $r=N-s$ differs from formula \eqref{eq:defr} from
the previous section.

It is easy to see that $\E$ is maximal if codimensions and multiplicities at singular points at finite distance are
minimal possible. So we put $m_1=\dots=m_s=3$, $m_{s+1}=\dots=m_N=2$, $\en_1=\dots=\en_s=3$, $\en_{s+1}=\dots=\en_N=1$
in $\E$. So we obtain
\[
\E=2r+6s+p'\nu'_\infty.
\]
From \eqref{eq:bmy3} we know that
\begin{equation}\label{eq:B02}
\nu'_\infty\le p+q+4g-\frac{11}{6}r-\frac72s-2.
\end{equation}
This formula differs from \eqref{eq:B01} by the term $-\frac12$ that is absent in \eqref{eq:B02}
(we have $-2$ instead of $-\frac72$). This diffenence
comes from the fact that there are $r+s+1$ singular points in the previous case and $r+s$ singular points in this case.

Substituting \eqref{eq:B02} into $\E$, and assuming an equality in \eqref{eq:B02} we get
\begin{equation}\label{eq:E20}
\begin{split}
\E=&2r+6s+p'(p+q+4g-\frac{11}{6}r-\frac72s-2)=\\=&2N+4s+p'(p+q+4g-\frac{11}{6}N-\frac53s-2).
\end{split}
\end{equation}
Obviously $\frac{\partial \E}{\partial s}=4-\frac53p'<0$ as $p'\ge 3$. Therefore
\[\E\le \E_6=\E|_{s=0}=2N+p'(p+q+4g-\frac{11}{6}N-2).\]
Define $\Delta_6=\D-\E_6$. We obtain
\[\Delta_6=(p-1)(q-1)-p'+1-p'(p+q+4g-\frac{11}{6}N-2)-2g-2N.\]
Clearly $\Delta_6$ is increasing with $q$. Thus
\begin{align*}
\Delta_6\ge\Delta_7=\Delta_6|_{q=p+p'}=&(p-1)(p+p'-1)-p'(2p+p'+4g-\frac{11}{6}N-2)-\\
&-2g-2N-p'+1.
\end{align*}
But $p'\to\Delta_7(p')$ is a concave function. So $\Delta_7(p')\ge\min(\Delta_7(3),\Delta_7(p/2))$ as $p'\in[3,p/2]$.
With $p'=3$ we obtain.
\[\Delta_7(3)=p^2-5p-14g+\frac72N-7.\]
Recall that by \eqref{eq:trivial1} $p\ge 6$. 
Therefore $\Delta_7(3)$ grows with $p$. Putting $p\ge N-2g+1$ we obtain
\[\Delta_7(3)\ge\Delta_8=(N-2g+\frac14)^2-7g-\frac{177}{16}.\]
Using \eqref{eq:MT1D} we obtain $\Delta_7(3)>0$.

On the other hand let
\[\Delta_9=\Delta_7(\frac{p}{2})=\frac14p^2-p(2g-\frac{11}{6}N+2)-2g-2N+2.\]
Then
\[2\frac{\partial\Delta_9}{\partial p}=p-4g+\frac{11}{6}N-4\stackrel{p\ge N-2g+1}{\ge}\frac{17}{6}N-6g-3.\]
The latter part is positive by \eqref{eq:MT1E}. So 
\[\Delta_9\ge\Delta_{10}=\Delta_9|_{p=N-2g+1}=
\frac76\left(N^2-(\frac{29}7g+\frac{31}{14})N+\frac{30}{7}g^2-\frac{6}{7}g+\frac{3}{14}\right).\]
But by \eqref{eq:MT1F} we have $\Delta_{10}>0$. This ends the proof of Theorem 1.
\begin{remark}\label{ong=1}
If $g=1$, the inequality \eqref{eq:MT1D} is not satisfied for $N=6$ because $\frac{7}{4}+\sqrt{7+\frac{177}{16}}=6$.
This is a slight problem, since a genus~$1$ curve with $6$ singular points at finite distance would violate
the Zaidenberg--Lin conjecture. But if $N=6$, and $g=1$ then $\Delta_7(3)=p^2-5p$. But by \eqref{eq:trivial1}
$\Delta_7(3)>0$ for $g=1$ and $N=6$. This case has not been rejected directly, because for small $N$ and $g$,
the inequality $p\ge N-2g+1$ may be weaker than $p\ge 6$.
\end{remark}
\section{Proof of Theorem 2}\label{Sec4}
The proof goes along the lines of the proof of Theorem 1. It has however an additional
ingredient: curves are allowed to have more branches at a given singular point.

Namely let us observe that the quantity $\E$ defined in \eqref{eq:mathcalE} is maximal 
if precisely one term is dominating, as in the proof of Theorem 1. We have
then two cases depending on whether this dominating term comes from singularity at finite
distance, or it is the term $p'\nu'_\infty$ that is the contribution from infinity.

\subsection{When singularities at finite distance dominate}\label{3.2}

In this case it is a trivial observation that $\E$ is maximal if $\en_1$ is as large as
possible and $\en_2,\dots,\en_N,\nu'$ are minimal. Let us discuss the minimal possible
values of $\en_i$
depending on the type of the singular point $z_i$.
\begin{itemize}
\item[1] $z_i$ is a unibranched singular point with multiplicity $2$. Then we
may assume that $\en_i=1$ and $\eta_i\ge\frac56$. $r$ will
denote the number of these points.
\item[2] $z_i$ is a unibranched singular point with multiplicity at least three. We shall
assume that $\en_i=3$, $m_i=3$ and $\eta_i>\frac12$. Such a point will contribute
$m_i(\en_i-m_i+2)=6$ to the sum $\E$. We will assume that there are precisely $s$ such
points.
\item[3] $z_i$ has $r_i>1$ branches. Then $\en_i$ must be larger than $r_i-2$. We shall
assume that $\en_i=r_i-2$ and $m_i=r_i$, what corresponds to an ordinary $r_i$-tuple point.
Such point will give a contribution of $m_i(\en_i-m_i+r_i+1)=r_i(r_i-1)$ to $\E$ (cf. Example~\ref{ordinary}). 
For fixed $r_i$ we shall assume that there are exactly $k_{r_i}$ such points.
\end{itemize}
Altogether we have $1+r+s+\sum k_i=N$ singular points.
Let us also denote
\begin{equation}
A:=\max_{1\le i\le N}r_i\label{eq:defA}
\end{equation}
the maximal total number of branches of singular points at finite distance.
\begin{remark}\label{newrem}
Since $m_i\ge r_i$ (multiplicity is never smaller than the number of branches), we have $\mx m_i\ge A$. Therefore $m_1\ge A$.
\end{remark}
Using the above notation we obtain
\[
\E=2r+6s+\sum_{j=2}^A k_jj(j-1)+m_1(\en_1-m_1+r_1+1).
\]
In the sequel we shall use the notation $B=r_1$.
Let also
\begin{equation}\label{eq:defTi}
T_n:=\sum_{j=2}^A k_j j^n.
\end{equation}
We will try to maximise the quantity 
\begin{equation}\label{eq:1E0}
\E=2r+6s+T_2-T_1+m_1(\en_1-m_1+B+1)
\end{equation}
under constrains
\begin{subequations}
\begin{align}
T_1-T_0+B-1&=R&\textrm{see \eqref{eq:defR}}\label{eq:SE1}\\
T_0+r+s+1&=N\label{eq:SE2}\\
m_1-B+r+2s&\le p+2g-1&\textrm{see \eqref{eq:multbound}}\label{eq:SE3}\\
\en_1+T_1-2T_0&\le p+q-2+R+4g-\frac{11}{6}r-\frac72s.&\textrm{see \eqref{eq:bmy3}}
\label{eq:SE4}
\end{align}
\end{subequations}
We shall assume that there is an equality in \eqref{eq:SE4}. As $T_1-2T_0=R-N+2+r+s-A$
by \eqref{eq:SE1} and \eqref{eq:SE2} we have
\begin{equation}
\label{eq:SE5}
\en_1=p+q+4g+N-4-\frac{17}{6}r-\frac92s+B
\end{equation}
\begin{proposition}\label{skomplikowane}
The quantity $\E$ in \eqref{eq:1E0} is maximal under constrains \eqref{eq:SE1}---\eqref{eq:SE4} if the following
conditions are satisfied:
\begin{itemize}
\item[(a)] there is an equality in \eqref{eq:SE3};
\item[(b)] $B=A$.
\item[(c)] $s=0$;
\item[(d)] $k_2=R-1$, $A=2$ and $k_i=0$ for $i\ge 3$;
\item[(e)] $r=N-R-1$.
\end{itemize}
\end{proposition}
\begin{proof}
The proof is split into several, mostly trivial, lemmas.
\begin{lemma}\label{poprzednilemat}
$\E$ grows with $m_1$, if $m_1$ satisfies \eqref{eq:SE3}.
\end{lemma}
\begin{proof}
As $m_1\le p+2g+B-r-2s-1\le\frac12(\en_1+B+1)$, the function $m_1\to(\en_1+B+1-m_1)$ is increasing in $m_1$.
\end{proof} 
According to the lemma we put
\[m_1=p+2g-1+B-r-2s.\]
Therefore, from \eqref{eq:1E0}:
\begin{equation}\label{eq:1E2}
\begin{split}
\E=&(p+2g-1+B-r-2s)(q+2g+B-\frac{11}{6}r-\frac52s+N-2)+\\
&+2r+6s+T_2-T_1.
\end{split}
\end{equation}
\begin{lemma}\label{B=A}
$\E$ is optimal if $B=A$.
\end{lemma}
\begin{proof}
Assume that, say for the singular point $z_2$ we have $r_2>B$ branches. The contribution from points $z_1$ and $z_2$ into
$\E$ is equal
\[c:=m_1(\en_1-m_1+r_1+1)+r_2^2-r_2.\]
Consider singular points $z_1'$ and $z_2'$ with the following parameters: $m_1'=m_1+1$, ${\en_1}'=\en_1+1$, $r_1'=r_1+1$,
$r_2'=r_2-1$ and $z_2'$ is an ordinary $(r_2-1)$-tuple point. The contribution from $z_1'$ and $z_2'$ into $\E$
is equal to
\[c'=(m_1+1)(\en_1-m_1+r_1+2)+r_2^2-3r_2+2=\en_1+r_1+5-2r_2+c.\]
But $\en_1+r_1+1\ge 2m_1$ by the proof of Lemma~\ref{poprzednilemat}. Then $c'-c\ge 4+2m_1-2r_2$. But $m_1\ge m_2=r_2$
by assumption that the singular points are ordered. Hence such change of the number of branches leads to an increment of $\E$.
\end{proof}
\begin{remark}
Now we are playing only with inequalities. We are not interested, at least in this paper, whether from the existence of
curve with singular points $z_1,\dots,z_N$ of given type, one can deduce the existence of curve with points of type 
$z_1',z_2',z_3,\dots,z_N$. The answer in general case is apparently negative.
\end{remark}
\begin{remark}
At the beginning of this section we claimed that $\E$ is maximal if $\en_1$ is maximal possible and all other
codimensions are minimal. A rigorous proof of this claim could follow the lines of the proof of Lemma~\ref{B=A}.
\end{remark}
\begin{lemma} If $N$, $R$ and $k_i$ are fixed, so only $r$ and $s$ are allowed to vary, then
$\E$ decreases with $s$.
\end{lemma}
\begin{proof}
By \eqref{eq:SE2} $-r=-N+T_0+s+1$. Therefore
\begin{align*}
\E=&(p+2g+A+T_0-N-s)(q+2g+A-\frac56N+\frac{11}{6}T_0-\frac23s-\frac{1}{6})+\\
&+2N+4s-2+T_2-T_1-T_0,
\end{align*}
where we have written $A$ instead of $B$ according to Lemma~\ref{B=A}.
Thus
\[\frac{\partial\E}{\partial s}=\frac43s-\frac23(p+2g+A+T_0-N)-
(q+2g+A-\frac56N+\frac{11}6T_0-\frac{1}6)+4.\]
On the other hand, combining \eqref{eq:SE2} and \eqref{eq:SE3} we obtain
\[s\le p+2g+T_0+A-N.\]
Therefore 
\[\frac{\partial\E}{\partial s}\le\frac{19}{6}-\frac13s-\frac53A-\frac56T_0-\frac16N.\]
But $A\ge2$, so $\frac{\partial\E}{\partial s}<0$.
\end{proof}
Putting $s=0$ in \eqref{eq:1E2} we obtain. 
\begin{multline}\label{eq:1E3}
\E\le\E_1=\E|_{s=0}=\\
=(p+2g-1+A-r)(q+2g+A-\frac{11}{6}r+N-2)+2r+T_2-T_1.
\end{multline}
\begin{lemma}\label{nowylemat}
If we keep $N$, $R$, $A$ and $r$ fixed then $\E_1$ is maximal
if $k_3=\dots=k_{A-1}=0$.
\end{lemma}
\begin{proof}
The dependency of $\E$, $N$ and $R$ on $k_i$ is hidden in quantities $T_j$ (see \eqref{eq:defTi}). Observe that
if we want to keep $N$ and $R$ fixed, we must fix precisely $T_1$ and $T_0$ by \eqref{eq:SE1} and \eqref{eq:SE2}.

Assume that $2\le x<y<z\le A$ are integers and we apply the change
\begin{align*}
k_x\to& k_x+\delta_x=k_x+\delta_z\frac{z-y}{y-x}\\
k_x\to& k_y+\delta_y=k_y+\delta_z\frac{x-z}{y-x}\\
k_x\to& k_z+\delta_z=k_z+\delta_z\frac{y-x}{y-x}
\end{align*}
with $\delta_z>0$. We have $\delta_x>0$ and $\delta_y<0$. Then $T_0$ and $T_1$ are obviously fixed and 
\[T_2\to T_2+\frac{\delta_z}{y-x}(zx(x-z)+zy(z-y)+xy(y-x))\]
and $zx(x-z)+zy(z-y)+xy(y-x)>0$ (this is left as an exercise).
Putting $x=2$ and $z=A$ we can then make the above change for any $3\le y\le A-1$ with $\delta_y=-k_y$.
Then $T_2$ will increase and so $\E_1$.
\end{proof}
Let us now assume that $k_3=\dots=k_{A-1}=0$. Then
\begin{align*}
T_0&=k_2+k_A\\
T_1&=2k_2+Ak_A.
\end{align*}
So $T_2-T_1=(A+1)T_1-2AT_0$. Expressing $T_1$ and $T_0$ with the help of \eqref{eq:SE1} and \eqref{eq:SE2}
yields
\[T_2-T_1=A(R-N+r-A+2)+(R+N-A-r).\]
Substituting this into $\E_1$ in \eqref{eq:1E3} we get
\begin{equation}\label{eq:1E4}
\begin{split}
\E_1\le\E_2=&(p+2g+A-r-1)(q+2g+A+N-\frac{11}{6}r-2)+\\
&+A(R-N+r+1)+(R+N-r)-A^2.
\end{split}
\end{equation}
\begin{lemma}\label{nowy2}
$\E_2$ is optimal if $k_i=0$ for $i\ge 3$.
\end{lemma}
\begin{proof}
Keeping $A,R$ and $N$ fixed we shall try to optimise the number of unicuspidal singular points with multiplicity $2$.
Let us differentiate \eqref{eq:1E4} with respect to $r$:
\[
\frac{\partial\E_2}{\partial r}=\frac{11}{3}r-\frac{11}{6}(p+2g+A-1)-(q+2g+A+N-2)+A+1.\]
But $r\le p+2g-1$ by \eqref{eq:SE3} so this derivative is bounded by
\[\frac56r-\frac{11}{6}A-N+1\]
As $N\ge r+1$ by \eqref{eq:SE2} and $A\ge 2$ we get $\frac{\partial\E_2}{\partial r}<0$. Therefore we must put
$r$ as small as possible. By \eqref{eq:SE2} this is the same as putting $T_0$ as large as possible, when $T_1-T_0$
is kept fixed (by \eqref{eq:SE1}). Consider the change
\begin{align*}
k_2&\to k_2+\delta_2=\frac{A-1}{A-1}\delta_2\\
k_A&\to k_A+\delta_A=-\frac{2-1}{A-1}\delta_2.
\end{align*}
Then $T_1-T_0=(A-1)k_A+k_2$ is fixed. On the other hand $T_0=k_2+k_A\to T_0+\frac{A-2}{A-1}\delta_2$. Therefore
$T_0$ is maximal if $k_2$ is maximal and $k_A=0$. 
\end{proof}
Assuming $k_i=0$ for $i\ge 3$ we obtain $T_1=2T_0$
 so by \eqref{eq:SE1} and \eqref{eq:SE2}
\begin{equation}\label{eq:1E4.5}
r=N+A-R-2.
\end{equation}
And
\[T_2-T_1=2T_0=2R+2-2A.\]
Substituting the two above quanities into \eqref{eq:1E4} yields
\begin{equation}\label{eq:1E5}
\begin{split}
\E_2\le\E_3=&(p+2g+1+R-N)(q+2g+\frac{11}{6}R-\frac56A-\frac56N+\frac53)+\\
&+2N-2.
\end{split}
\end{equation}
Obviously $\E_3$ decreases with $A$. The minimal value of $A$ is $2$. So we get
\begin{equation}\label{eq:1E6}
\E_3\le\E_4=\E_3|_{A=2}=(p+2g+1+R-N)(q+2g+\frac{11}{6}R-\frac56N)+2N-2.
\end{equation}
The proposition is now proved and $\E\le\E_4$ depends only on $p$, $q$, $g$, $R$ and $N$.
\end{proof}
The remaining part of the proof of Theorem 2 follows the proof of Theorem~1. Define
\[\Delta_4=(p-1)(q-1)-q+p-2g+1-\E_4.\]
As $(p-1)(q-1)+q-p-2g\le\D$ (see \eqref{eq:Milnor2}) and $\E_4\ge\E$, we have $\Delta_4\le\Delta$. To complete
the proof it suffices to show that $\Delta_4>0$.
Differentiating $\Delta_4$ with respect to $q$ yields
\[\frac{\partial\Delta_4}{\partial q}=-2g-3-R+N.\]
By \eqref{eq:MT2A} the latter expression is non--negative. Thus
\[\Delta_4\ge\Delta_5=\Delta_4|_{q=p+1}.\]
More precisely 
\begin{multline*}
\Delta_5=p\left(\frac{11}{6}N-\frac{17}{6}R-4g-3\right)-\\
-\frac{11}{6}R^2+\frac83NR-\frac{17}{3}gR-\frac56N^2+\frac{11}{3}gN-4g^2-
\frac{17}{6}R-\frac16N-6g-1.
\end{multline*}
By \eqref{eq:MT2B} $\Delta_5$ is growing with $p$. The minimal value of $p$ is given by \eqref{eq:SE3}: it is $r-2g+1$.
But $r=N-R$ by \eqref{eq:1E4.5} and the fact that $A=2$. So $p\ge N-R-2g+1$.
Substituting this into $\Delta_5$ yields
\[\Delta_5\ge\Delta_6=\Delta_5|_{p=N-R-2g+1}.\]
After straightforward computations we obtain
\[\Delta_6=\left(N-2g-R-\frac23\right)^2-\left(\frac{20}{3}g+4R+\frac{22}{9}\right).\]
So by \eqref{eq:MT2C} we have $\Delta_6>0$.
\subsection{When $p'$ is dominating.}
Here $\mathcal{E}$ will be maximal when $\nu'_\infty$ is maximal and all other codimensions are minimal possible. 
Assume that we have (compare beginning of Section~\ref{3.2})
\begin{itemize}
\item $r$ unibranched singular points with multiplicity $2$;
\item $s$ unibranched singular points with multiplicity $3$;
\item $k_j$ ordinary $j-$tuple points.
\end{itemize}
Then using $T_i$ as in \eqref{eq:defTi} we can write
\[\E=2r+6s+T_2-T_1+p'\nu'_\infty.\]
And all the variables $p$, $q$, $g$, $k_i$, $r$, $s$, $R$, $A$ and $N$ are subject to contrains similar to 
\eqref{eq:SE1}\dots\eqref{eq:SE4}. Namely
\begin{subequations}
\begin{align}
T_1-T_0&=R\label{eq:NE1}\\
r+s+T_0&=N\label{eq:NE2}\\
r+2s&\le p+2g-1\label{eq:NE3}\\
\nu'_\infty+T_1-2T_0&\le p+q-2+R+4g-\frac{11}{6}r-\frac72s\label{eq:NE4}.
\end{align}
\end{subequations}
The difference between \eqref{eq:NE1}\dots\eqref{eq:NE3} and \eqref{eq:SE1}\dots\eqref{eq:SE3} lies in the fact, that
in the previous section we had a distinguished singular point with multiplicity $m_1$ and $B$ branches.

\begin{lemma}
Under contrains \eqref{eq:NE1}\dots\eqref{eq:NE4}, with fixed $T_i$, $\E$ is maximal if $s=0$.
\end{lemma}
\begin{proof}
By \eqref{eq:NE2} $r=N-T_0-s$. Then
\[\nu'_\infty\le p+q-2+R+4g+\frac{23}{6}T_0-T_1+\frac{11}{6}T_1-\frac53s.\]
As $2r+6s=2N-2T_0+4s$, $\frac{\partial E}{\partial s}=4s-\frac53p'$. But $p'\ge 3$.
\end{proof}
Using the above lemma we get
\[\E\le\E_7=\E|_{s=0}=2r+T_2-T_1+p'(p+q-2+R+4g-\frac{11}{6}r+2T_0-T_1).\]
In other words, using  \eqref{eq:NE1} and \eqref{eq:NE2} we get $2T_0-T_1=N-R-r$ so
\begin{equation}\label{eq:1E7}
\E_7=2r+T_2-T_1+p'(p+q-2+4g-\frac{17}{6}r+N).
\end{equation}
By Lemma~\ref{nowylemat} we can assume that $k_3=\dots=k_{A-1}=0$.
Thus $T_2-T_1=(A+1)T_1-2AT_0=(A+1)R+(1-A)N+(A-1)r$. So
\[\E_7\le\E_8=A(R-N+r)+R+N+r+p'(p+q-2+4g-\frac{17}{6}r+N).\]
\begin{lemma}\label{geom}
We have $p'>A$.
\end{lemma}
\begin{proof}
If $p'\le A$, the term $p'\nu'_\infty$ is not dominating in $\E$, because some multiplicity of a singular point (say $z_1$) 
at finite distance is larger than $p'$. Then increasing $\en_1$ at the cost of decreasing $\nu'_\infty$ makes $\E$ grow.
\end{proof}
From this lemma we conclude that $\frac{\partial\E_8}{\partial r}<0$. Using essentially the same arguments as in the proof 
of Lemma~\ref{nowy2} we infer that $A=2$ and $k_i=0$ for $i\ge 3$. Then $T_1=2T_0$ so
\begin{equation}\label{eq:1E8}
r=N-R.
\end{equation}
The meaning of this formula is clear: $N$ is the number of all singular points, under assumption that $A=2$, $R$ becomes
the number of singular points with $2$ branches and $r$ is the number of singular points with one branch, because
$s=0$.

Substituting \eqref{eq:1E8} into $\E_8$ we obtain
\[\E_8\le\E_9=2N+p'(p+q-2+4g+\frac{17}{6}R-\frac{11}{6}N).\]
The appearance of $2N$ in $\E_9$ is not surprising, for any unibranched singular point contributes $2$ and
any ordinary double point contributes $2$. So $2N$ is the contribution into $\E$ of all singular points at finite distance.

Now define
\[\Delta_9=(p-1)(q-1)-p'+1-2g-\E_9.\]
Obviously $\Delta\ge\Delta_9$. We shall strive to show that $\Delta_9>0$. 
We see that $\Delta_9$ grows with $q$ so
\begin{align*}
\Delta_9\ge\Delta_{10}=\Delta_9|_{q=p+p'}=&(p-1)(p+p'-1)+1-2g-2N\\
&-p'(2p+p'-1+4g+\frac{17}{6}R-\frac{11}{6}N).
\end{align*}
But $p'\to\Delta_{10}(p')$ is a concave function. If $p'\in[3,p/2]$ (by Lemma~\ref{geom} $p'>A\ge 2$), 
$\Delta_{10}$ attains its minimum at the boundary
of this interval. 

\underline{\textit{Estimating $\Delta_{10}(3)$.}} Assuming $p'=3$ we get
\[\Delta_{10}(3)=p^2-5p-14g-\frac{17}{2}R+\frac{7}{2}N-7.\]
As $p'\ge 3$, $\Delta_{10}(3)$ grows with $p$. Putting $p\ge N-R-2g+1$ we get
\[\Delta_{10}(3)\ge\Delta_{11}=(N-2g-R+\frac14)^2-7g-5R-\frac{177}{16}.\]
By \eqref{eq:MT2D} we get $\Delta_{11}>0$.

\smallskip
\underline{\textit{Estimating $\Delta_{10}(p/2)$.}} Assuming $p'=p/2$ we get
\[\Delta_{10}(p/2)=\Delta_{12}=\frac14p^2+p(-2g-\frac{17}{12}R+\frac{11}{12}N-2)-2g-2N+2.\]
Then $\frac{\partial\Delta_{11}}{\partial p}=\frac12p-2g-\frac{17}{12}R+\frac{11}{12}N-1$.
By \eqref{eq:NE3} and \eqref{eq:1E8} we have $p\ge r-2g+1=N-R-2g+1$,so
\[\frac{\partial\Delta_{12}}{\partial p}\ge \frac{17}{12}(N-\frac{36}{17}g-\frac{23}{17}R-\frac{18}{17}).\]
Using \eqref{eq:MT2E} we get that this derivative is non--negative.
Hence
\begin{multline*}
\Delta_{12}\le\Delta_{13}
=\frac76\left(\left(N-\frac{29}{14}g-\frac{17}{14}R-\frac{31}{28}\right)^2-\right.\\
-\left.\left(\left(\frac{1}{14}g+\frac{3}{14}R\right)^2+\frac{1067}{196}g
+\frac{513}{196}R+\frac{793}{784}\right)\right).
\end{multline*}
By \eqref{eq:MT2F} $\Delta_{13}>0$. The proof is completed
\begin{remark}\label{ong=0}
As it was mentioned at the end of Section~\ref{S12}, for $g=0$, $R=1,2$ the bounds \eqref{eq:MT2C}, \eqref{eq:MT2D}
and \eqref{eq:MT2F} are unsatisfactory. The reason is the same as observed in Remark~\ref{ong=1}: for small $g$ and $R$,
the bound $p\ge N-R-2g+1$ is weaker than the bound $p\ge 6$. Repeating the arguments of this Section for $g=0$ and $R=1,2$,
using the inequality $p\ge 6$ instead would show that $N\le 3$ for $g=0,R=1$ and $N\le 5$ for $g=0,R=2$. We do
not give the straightforward prove here.
\end{remark}
\begin{remark}
The method presented in this paper can be applied to bound the number of singular points of an arbitrary algebraic curve in $\CC^2$.
We plan to investigate it in subsequent papers.
\end{remark}
\begin{acknowledgements}
The author expresses his thanks to Henryk \.Zo\l\c{a}dek for stimulating discussions. Part of the work has been
completed during the stage of the author at McGill University in Montreal. The author is grateful to Peter Russell
for the invitation.
\end{acknowledgements}

\end{document}